\documentclass[11pt,a4paper]{article}
\pdfoutput=1 % for arxiv

\newcommand{\op}[1]{\textcolor{red}{OP: #1}}

\usepackage{epsf,epsfig,amsfonts,amsgen,amsmath,amstext,amsbsy,amsopn,amsthm,cases,listings,color
}
\usepackage{ebezier,eepic}
\usepackage{color}
\usepackage{multirow}
\usepackage{epstopdf}
\usepackage{graphicx}
\usepackage{pgf,tikz}
\usepackage{mathrsfs}
\usepackage[marginal]{footmisc}
\usepackage{enumitem}
\usepackage[titletoc]{appendix}
\usepackage{booktabs}
\usepackage{url}
\usepackage{relsize}
\usepackage{mathtools}
\usepackage{comment}
\usepackage{pgfplots}
\usepackage{amssymb}
\usepackage{float}
\usepackage{wasysym}

\usepackage{empheq}

\usepackage{dsfont}

\usepackage{tikz}
\usepackage{longtable}
\usepackage{subfigure}
\pgfplotsset{compat=1.18}
\usepackage{mathrsfs}
\usepackage{wasysym}
\usetikzlibrary{arrows}
\usepackage{aligned-overset}
\usepackage{bm}%for vector
\usepackage{bbm}%for vector
\usepackage[T1]{fontenc}%for polish hook
\usepackage[backref=page]{hyperref}
\allowdisplaybreaks[1]

\definecolor{uuuuuu}{rgb}{0.27,0.27,0.27}
\definecolor{sqsqsq}{rgb}{0.1255,0.1255,0.1255}

\newtheorem{definition}{Definition} [section]
\newtheorem{theorem}[definition]{Theorem}
\newtheorem{lemma}[definition]{Lemma}

\newtheorem{claim}[definition]{Claim}

\theoremstyle{remark}
\newtheorem{remark}[definition]{Remark}

\newenvironment{pf}{\noindent{\bf Proof}}

\newcommand{\B}[1]{\mathbf{#1}}

\def\PP{\mathbb{P}}

\newcommand{\dd}{\mathrm{d}}
\newcommand{\hide}[1]{}

\newcommand{\vol}{\operatorname{vol}}

\newcommand{\me}{\mathrm{e}}

\begin{document}
\title{\bf\Large Covering large-dimensional Euclidean spaces by random translates of a given convex body}
\date{\today}

\author{
Boris Bukh\thanks{
Department of Mathematical Sciences, Carnegie Mellon University, USA.
Research supported in part by a Simons Foundation Fellowship, and NSF grants
DMS-2154063 and DMS-2452120.
Email: \texttt{bbukh@math.cmu.edu}
}
\and
Jun Gao\thanks{
Academy of Mathematics and Systems Science, Chinese Academy of Sciences,
Beijing 100084, China.
Research supported by ERC Advanced Grant 101020255.
Email: \texttt{jungao@amss.ac.cn}
}
\and
Xizhi Liu\thanks{
School of Mathematical Sciences, University of Science and Technology of China,
Hefei, China.
Research supported by ERC Advanced Grant 101020255 and the Excellent Young
Talents Program (Overseas) of the National Natural Science Foundation of China.
Email: \texttt{liuxizhi@ustc.edu.cn}
}
\and
Oleg Pikhurko\thanks{
Mathematics Institute and DIMAP, University of Warwick, Coventry, UK.
Research supported by ERC Advanced Grant 101020255.
Email: \texttt{\{o.pikhurko,shumin.sun\}@warwick.ac.uk}
}
\and
Shumin Sun\footnotemark[4]
}

\maketitle
%%%%%%%%%%%%%%%%%%%%%%%%%%%%%%%%%%%%%%%%%%%%%
\begin{abstract}
    %We show that there exists a covering of $\mathbb{R}^{n}$ by unit balls with density at most $\left(1/2+o(1) \right)n \ln n$, in which every point is covered at most $\left(1.79556... + o(1)\right) n \ln n$ times.  
    %This improves a classical result of Erd\H{o}s and Rogers [\emph{Acta Arith.} \textbf{7}, 1961/62], whose construction has density $\left(1 + o(1)\right) n \ln n$ and multiplicity $\left(\mathrm{e} + o(1)\right) n \ln n$.
    %We also show that the constant $1/2$ cannot be improved using approaches based on lattice points combined with random points sampled from a Poisson point process. 

Determining the minimum density of a covering of $\mathbb{R}^{n}$ by Euclidean unit balls as $n\to\infty$ is a major open problem, with the best known results being the lower bound of $\left(\mathrm{e}^{-3/2}+o(1)\right)n$ by Coxeter, Few and Rogers~[\emph{Mathematika} \textbf{6}, 1959] and the upper bound of $\left(1/2+o(1) \right)n \ln n$ by Dumer~[\emph{Discrete Comput.\ Geom.} \textbf{38}, 2007]. 

We prove that there are ball coverings of $\mathbb{R}^n$ attaining the asymptotically best known density $\left(1/2+o(1) \right)n \ln n$ such that, additionally, every point of $\mathbb{R}^n$ is covered at most $\left(1.79556... + o(1)\right) n \ln n$ times. This strengthens the result of Erd\H{o}s and Rogers [\emph{Acta Arith.}\ \textbf{7}, 1961/62] who had the maximum multiplicity at most $\left(\mathrm{e} + o(1)\right) n \ln n$.

On the other hand, we show that the method that was used for the best known ball coverings (when one takes a random subset of centres in a fundamental domain of a suitable lattice in $\mathbb{R}^n$ and extends this periodically) fails to work if the density is less than $(1/2+o(1))n\ln n$; in fact, this result remains true if we replace the ball by any convex body~$K$. Also, we observe that a ``worst'' convex body $K$ here is a cube, for which the covering density coming from random constructions is only $(1+o(1))n\ln n$.
\end{abstract}

%%%%%%%%%%%%%%%%%%%%%%%%%%%%%%%%%%%%%%
\section{Introduction}\label{SEC:Introduction}
%%     n \ln n + n \ln\ln n + 5n. 
The fundamental result of Rogers~\cite{Rog57A} states that, for every \emph{convex body} $K$ (i.e.\ a compact convex set with non-empty interior) in the $n$-dimensional Euclidean space $\mathbb{R}^n$, there exists a covering of $\mathbb{R}^{n}$ by translates of $K$ with density at most $(1 + o(1)) n \ln n$ as $n\to\infty$. If $K$ is a ball then the covering density was improved  to $(1/2 + o(1)) n \ln n$ by Dumer~\cite{07D} while the best known lower bound is $\left(\mathrm{e}^{-3/2} + o(1)\right)n$ by Coxeter, Few and Rogers~\cite{CFR59}.
%who improved upon earlier results of Bambah--Davenport~\cite{BD52} and  Erd\H{o}s--Rogers~\cite{ER53}. 
Closing this gap for ball covering is a major open problem. 

In brief, the new results proved in this paper are as follows. First, Theorem~\ref{THM:covering-multiplicity} proves that  there are ball coverings of $\mathbb{R}^n$ attaining the asymptotically best known density $\left(1/2+o(1) \right)\cdot n \ln n$ such that, additionally, every point of $\mathbb{R}^n$ is covered at most $\left(1.79556... + o(1)\right) n \ln n$ times. This strengthens a result of Erd\H{o}s and Rogers~\cite{62ER} who had the maximum multiplicity at most $\left(\mathrm{e} + o(1)\right) n \ln n$. Also, Theorem~\ref{THM:poisson-process-lower-bound} shows that, for any given convex body $K\subseteq \mathbb{R}^n$, the method that was used in~\cite{Rog57A,07D}
%for the best known $K$-coverings 
(when one takes a random subset of centres in a fundamental domain of a suitable lattice in $\mathbb{R}^n$ and extends this periodically) fails with high probability to produce a $K$-covering if the appropriately defined density is strictly less than $(1/2+o(1))n\ln n$. 
%In fact, this result remains true if we replace the ball by any convex body $K$, see Theorem~\ref{THM:poisson-process-lower-bound}. 
Finally, Theorem~\ref{THM:poisson-process-lower-bound-cube}  gives that, somewhat surprisingly,  a ``worst'' convex body $K$ for this method is a cube, for which the covering density coming from random constructions is only $(1+o(1))n\ln n$ (which matches the bound of Rogers~\cite{Rog57A}).

%As a natural extension of Rogers’ result, Erd\H{o}s and Rogers~\cite{62ER} considered coverings of $\mathbb{R}^{n}$ by $K$ with small multiplicity, that is, each point is covered only a limited number of times. Their result shows that there exists a covering of $\mathbb{R}^{n}$ by translates of $K$ with density at most $(1 + o(1)) n \ln n$, in which every point in $\mathbb{R}^{n}$ is covered at most $(\me + o(1)) n \ln n$ times.

Let us provide the details. %We say $P$ is an \emph{$r$-ball covering} of $A$ if 
%\begin{align*}
%    A \subseteq \bigcup_{\mathbf{x} \in P} B_{r}^{n}(\mathbf{x}). 
%\end{align*}
%
%In this work, we focus on the problem of covering $\mathbb{R}^{n}$ with unit balls.
Let 
 $$
 B_r^n(\B x)\coloneqq\{\B y\in \mathbb{R}^{n}\colon \|\mathbf{x}-\mathbf{y}\|\le r\}
 $$
 denote the closed ball of radius $r$ centred at $\B x$, where $\|\cdot\|$ denotes the Euclidean norm.
When $\mathbf{x} = \mathbf{0}$ is the origin of $\mathbb{R}^n$, we write $B_{r}^{n}$ for brevity. Let $\vol(\cdot)$ denote the Lebesgue measure on~$\mathbb{R}^n$. 
%A ball of radius $1$ is called a \emph{unit ball}. 
Also, we define $\nu_{n}\coloneqq\vol(B_1^n)$ to be  the volume of the $n$-dimensional unit ball $B_{1}^{n}$. 
% Two points $\mathbf{x} \in \R^n$ and $\mathbf{x} \in \R^n$ are \emph{$d$-close} if their Euclidean distance, denoted $\mathrm{dist}(\mathbf{x}, \mathbf{y})$, is at most $d$. 
% We say that a subset $P \subseteq A$ is an \emph{$r$-ball packing} of $A$ if the distance between every pair of distinct points in $P$ is greater than $2r$. 

Let $K \subseteq \mathbb{R}^{n}$ be a convex body. 
A (discrete) set $P \subseteq \mathbb{R}^{n}$ is called a \emph{$K$-covering} of a set $A \subseteq \mathbb{R}^{n}$ if $A$ is a subset of 
\begin{align*}
   % A \subseteq 
   P + K 
    \coloneqq \left\{ \mathbf{p} + \mathbf{x} \colon \text{$\mathbf{p} \in P$ and $\mathbf{x} \in K$} \right\}, 
\end{align*}
 that is, the union of the translates of $K$ by the elements of $P$ contains every point of $A$.
For a measurable $A\subseteq \mathbb{R}^n$ of finite measure, 
the \emph{covering density} of a set $P\subseteq \mathbb{R}^n$ over $A$ is defined as 
\begin{align*}
    \vartheta_{K}(P, A)
    \coloneqq \sum_{\mathbf{p} \in P} \frac{\vol\big( A\cap (K+\mathbf{p}) \big)}{\vol(A)}. 
\end{align*}
The \emph{covering multiplicity of a point} $\mathbf{x} \in A$ is defined as 
\begin{align*}
    \mu_{K}(P, \mathbf{x})
    \coloneqq \left| \left\{ \mathbf{p} \in P \colon \mathbf{x} \in K + \mathbf{p} \right\} \right|. 
\end{align*}
The \emph{covering multiplicity of} $P$ is then defined as
\begin{align*}
    \mu_{K}(P, A)
    \coloneqq \sup\left\{ \mu_{K}(P, \mathbf{x}) \colon \mathbf{x} \in A\right\},
\end{align*}
i.e.,\ the maximum number of times that a point of $A$ is covered by the translates of $K$ by the elements of $P$. The \emph{$K$-covering density} and \emph{$K$-covering multiplicity} of $P$ over $\mathbb{R}^{n}$ are defined by 
\begin{align*}
    \vartheta_K(P)
    \coloneqq \limsup_{r\to \infty} \vartheta_K\big(P, [-r,r]^n\big) 
    \quad\text{and}\quad 
    \mu_K(P) 
    \coloneqq \sup_{r>0} \mu_K\big(P, [-r,r]^n\big). 
\end{align*}

In the case when $K$ is the unit ball $B_1^n$, we may omit the subscript $K$ and write just $\vartheta(P, A)$, $\mu(P, \mathbf{x})$, $\mu(P, A)$, etc.
%Let $\vol(\cdot)$ denote the Lebesgue measure of a set. 
%Suppose that $A$ is Lebesgue measurable with finite measure. 
%Define the \emph{covering density} of $P$ \emph{over} $A$ by 
%\begin{align*}
%    \vartheta(P, A)
%    \coloneqq \sum_{\mathbf{p} \in P} \frac{\vol\left( A\cap \big( B_{1}^{n}+\mathbf{p} \big) \right)}{\vol(A)}. 
%\end{align*}
%
%For every point $\mathbf{x} \in A$, the \emph{covering multiplicity} of $\mathbf{x}$ is given by 
%\begin{align*}
%    \mu(P, \mathbf{x})
%    \coloneqq \left| \left\{ \mathbf{y} \in P \colon \mathbf{x} \in B_{1}^{n} + \mathbf{y} \right\} \right|. 
%\end{align*}
%
%The \emph{covering multiplicity} of $P$ is then defined as 
%\begin{align*}
%    \mu(P, A)
%    \coloneqq \sup\left\{ \mu(P, \mathbf{x}) \colon \mathbf{x} \in A\right\}. 
%\end{align*}
%
Thus, for example,
%the classical Rogers bound shows that, there exists a unit ball covering $P$ of $\mathbb{R}^n$ such that $\vartheta(P)\le \left(1+o(1)\right)n \ln n$. 
Dumer's result~\cite{07D} gives a unit ball covering $P$ of $\mathbb{R}^n$  with $\vartheta(P)\le \left(1/2+o(1)\right)n \ln n$. Our first result shows that, additionally, one can assume that the covering multiplicity of $P$ is at most $(1.795...+o(1))n\ln n$, improving on the previous multiplicity bound $(\mathrm{e}+o(1))n\ln n$ of Erd\H{o}s and Rogers~\cite{62ER}, where $\mathrm{e}=2.718...$ is the base of the natural logarithm.
\begin{theorem}\label{THM:covering-multiplicity}
    There exists a unit ball covering $P$ of $\mathbb{R}^{n}$ such that 
    \begin{align*}
        \vartheta(P) \le \left(\frac{1}{2} + o(1)\right) n \ln n
        \quad\text{and}\quad 
        \mu(P) \le \big( \xi + o(1) \big) n \ln n,  
    \end{align*}
    where $\xi = 1.79556...$ is the unique (positive) real root of 
    \begin{align}\label{equ:xi-def}
         \ln \left(\frac{2x}{\me}\right) = \frac{1}{2x}. 
    \end{align}
\end{theorem}

Recall that a \emph{lattice} $\Lambda\subseteq \mathbb{R}^{n}$ is a set of the form
\begin{align*}
    \Lambda
    = \left\{\lambda_1 \mathbf{b}_{1} + \cdots + \lambda_n \mathbf{b}_{n} \colon \text{$\lambda_i \in \mathbb{Z}$ for $i \in [n]$}\right\},
\end{align*}
 for some linearly independent vectors $\mathbf{b}_{1}, \ldots, \mathbf{b}_{n}  \in \mathbb{R}^{n}$, where we use the shorthand $[n] \coloneqq\{1,\dots,n\}$.
For a convex body $K\subseteq \mathbb{R}^n$, we call a lattice $\Lambda \subseteq \mathbb{R}^{n}$ a \emph{packing lattice} of $K$ if, for every distinct elements $\B p,\B p'\in \Lambda$, the translates $K+\B p$ and $K+\B p'$ are disjoint.  
The corresponding quotient $\mathbb{T}_{\Lambda}\coloneqq\mathbb{R}^n/\Lambda$ (of the topological Abelian group $(\mathbb{R}^n,+)$ by the discrete subgroup $\Lambda$) is then called a \emph{packing torus} of~$K$. Slightly abusing notation, we  use $\vol(\cdot)$ also to denote the measure on $\mathbb{T}_{\Lambda}$ obtained by identifying the torus with a fundamental domain of~$\Lambda$ and using the Lebesgue measure on the latter. Likewise, we may use the same symbol $A$ to refer to the projection of $A\subseteq \mathbb{R}^n$ to $\mathbb{T}$, etc.

Recall that a discrete random variable $\mathbf{X}$ has the \emph{Poisson distribution with mean $\lambda > 0$}, denoted $\mathbf{X} \sim \mathrm{Pois}(\lambda)$, if for each integer $k\ge 0$ the probability of $\mathbf{X}$ assuming value $k$ is ${\lambda^{k} \me^{-\lambda}}/{k!}$.
\hide{
\begin{align*}
    \mathbb{P}\left[\,\mathbf{X} = k\,\right] = \frac{\lambda^{k} \me^{-\lambda}}{k!}, 
    \quad\text{for}\quad k \in \mathbb{N}. 
\end{align*}
}%
A random set $\mathbf{X} \subseteq A$ is said to be sampled from the \emph{Poisson point process} on measurable subset $A\subseteq \mathbb{R}^n$ with intensity $\rho$ if the following conditions hold:
\begin{enumerate}[label=(\roman*)]
    \item For every bounded measurable set $S \subseteq A$, the number of points in $S$, that is, $|\mathbf{X} \cap S|$, has the Poisson distribution with mean $\rho \cdot \vol(S)$. 
    \item For any disjoint bounded measurable sets $S_1,\dots,S_k\subseteq A$, the random variables $|\mathbf{X} \cap S_1|,\dots,|\mathbf{X} \cap S_k|$ are mutually independent.
\end{enumerate}
 Such a distribution is known to exist and to be uniquely determined by the above two properties.

As in many upper-bound constructions, the covering set $P \subseteq \mathbb{R}^{n}$ that we construct will be
of the form $\Lambda + X$, where $\Lambda \subseteq \mathbb{R}^{n}$ is a suitable lattice and $X$ is a typical (under some distribution) finite subset  of  the corresponding torus $\mathbb{T}_{\Lambda} \coloneqq \mathbb{R}^{n}/\Lambda$. Specifically,
Theorem~\ref{THM:covering-multiplicity} will follow from the following result. 
\begin{theorem}\label{THM:poisson-process-upper-bound}
    Let $0<\delta< 1/2$ be a real number and $M>0$ be a constant. Suppose that $\mathbb{T}$ is a packing torus of $B_{2}^{n}$ with $\vol(\mathbb{T}) \le M^n \nu_{n}$. 
    Suppose that the set $\mathbf{X}$ is sampled from the Poisson point process on $\mathbb{T}$ with intensity  
    \begin{align}\label{equ:rho-in-THM}
        \rho \coloneqq \left(\frac{1}{2} + \delta \right) \frac{n \ln n}{\nu_{n}}. 
    \end{align}
    Then the probability that all the following events occur tends to $1$ as $n \to \infty$: 
    \begin{enumerate}[label=(\roman*)]
        \item\label{THM:poisson-process-upper-bound-a} $\mathbb{T} \subseteq \mathbf{X} + B_{1}^{n}$, 
        \item\label{THM:poisson-process-upper-bound-b} $\vartheta(\mathbf{X}, \mathbb{T}) \le \left(\frac{1}{2} + 2\delta \right) n \ln n$, and 
        \item\label{THM:poisson-process-upper-bound-c} $\mu(\mathbf{X}) \le \left(\xi + 10 \delta \right) n \ln n$, where $\xi = 1.79556...$ is the unique real root of~\eqref{equ:xi-def}. 
    \end{enumerate}  
\end{theorem}

The derivation of Theorem~\ref{THM:covering-multiplicity} from Theorem~\ref{THM:poisson-process-upper-bound} is straightforward, given the known results on the existence of reasonably good packing lattices for balls, for which e.g.\ the classical Minkowski--Hlawka theorem suffices (and we refer the reader to Klartag~\cite{klar2025} and Abuya, Gargava and Zhao~\cite{AGZ} for the best known bounds in high dimensions). Also, note  the result of Rogers~\cite{R47_sphere} that, for any convex body $K\subseteq \mathbb{R}^n$,
there exists a packing lattice $\Lambda$ of $K$ such that $\vol\big(\mathbb{T}_{\Lambda}\big) \le 4^n\cdot \vol(K)$.

Our next result shows that the constant $1/2$ in~\eqref{equ:rho-in-THM} cannot be improved using this approach, in fact not only for a ball but for any convex body.
\hide{To state the result, we require some additional definitions. A lattice $\Lambda \subseteq \mathbb{R}^{n}$ is called a \emph{packing lattice of} $K$ if for every pair of distinct points $\mathbf{p}, \mathbf{p}' \in \Lambda$,  
\begin{align*}
    \big( K + \mathbf{p} \big) \cap \big( K + \mathbf{p}' \big) = \emptyset. 
\end{align*}
The torus $\mathbb{T}_{\Lambda}\coloneqq\mathbb{R}^n/\Lambda$ is then called a \emph{packing torus} of $K$. 
}

%\op{I fixed the def of $\rho$}
\begin{theorem}\label{THM:poisson-process-lower-bound}
    Let $\delta > 0$ be a real number. 
    Let $K \subseteq \mathbb{R}^{n}$ be a convex body and $\mathbb{T}$ be a packing torus of $K-K:=\{\mathbf{x}-\mathbf{y}\colon \mathbf{x},\mathbf{y}\in K\}$. 
    Suppose that $\mathbf{X} \subseteq \mathbb{T}$ is sampled from the Poisson point process with intensity 
    \begin{align*}
        \rho \coloneqq \frac{1}{\vol(K)}\left(\frac{n \ln n}{2} - (1+ \delta) n \ln\ln n \right). 
    \end{align*}
    Then the probability that $\mathbb{T} \subseteq \mathbf{X} + K$ tends to $0$ as $n\to \infty$.
\end{theorem}

Note that,  by %Conclusions~\ref{THM:poisson-process-upper-bound-a} and \ref{THM:poisson-process-upper-bound-b} of 
Theorem~\ref{THM:poisson-process-upper-bound}, the constant $1/2$ in front of $n\ln n$ in Theorem~\ref{THM:poisson-process-lower-bound} is best possible when $K$ is a unit ball and $\vol(\mathbb{T}) \le M^n \nu_{n}$, where $M$ is some constant.

\begin{remark}\label{re:N}
If we condition a Poisson point process on resulting in exactly $N$ points then we get $N$ independent uniform points in the domain.  Thus, by monotonicity and concentration 
(cf Claim~\ref{CLAIM:X-concentration}),
%(and by $\vol(\mathbb{T})\ge \vol(2K)=2^n$), 
the conclusion of Theorem~\ref{THM:poisson-process-lower-bound} also holds if $\mathbf{X}$ is obtained by choosing $N$ independent points, provided $N$ is strictly less than $(1/2+o(1))\,(n\ln n)\cdot \vol(\mathbb{T})/\vol(K)$. 
\end{remark}

\hide{
Also, note that Theorem~\ref{THM:poisson-process-lower-bound} is non-vacuous for any convex body $K\subseteq \mathbb{R}^n$: by a result of e.g.\ Rogers and Shephand~\cite{57RS} there exists a $K$-packing lattice $\Lambda\subseteq \mathbb{R}^n$ satisfying $\vol\big(\mathbb{T}_{\Lambda}\big) \le 4^n\cdot \vol(K)$. In the special case when $K$ is a ball $B^n_r$, a recent break-through result of Klartag~\cite{klar2025} shows that one can find a packing torus  $\mathbb{T}_{\Lambda}$ with $\vol\big(\mathbb{T}_{\Lambda}\big) \le cn^{-2}\cdot 2^n \vol(B_r^n)$.
}
Recall that the covering result of Rogers~\cite{Rog57A} holds for all convex bodies.
A natural question is whether the Euclidean ball in Theorem~\ref{THM:poisson-process-upper-bound}\ref{THM:poisson-process-upper-bound-a} can be replaced by an arbitrary convex body.
The following result shows that this is not the case. 

\begin{theorem}\label{THM:poisson-process-lower-bound-cube}
    Let $\delta > 0$ be a real number. 
    Let $C_{n}\coloneqq[-1/2,1/2]^n$ be the unit cube in $\mathbb{R}^{n}$ (of volume $1$) and $\mathbb{T}$ be a packing torus of $C_n-C_n$ (which is equal to $2 C_{n}$). 
    Suppose that $\mathbf{X} \subseteq \mathbb{T}$ is sampled from the Poisson point process with intensity 
    \begin{align*}
        \rho \coloneqq n \ln n - (1+ \delta) n \ln\ln n . 
    \end{align*}
    Then the probability that $\mathbb{T} \subseteq \mathbf{X} + C_n$ tends to $0$ as $n\to \infty$.
\end{theorem}

Similarly as in Remark~\ref{re:N}, the conclusion of Theorem~\ref{THM:poisson-process-lower-bound-cube} also holds if we sample fewer than $(1+o(1))n\ln n\cdot \vol(\mathbb{T})$ uniform points in $\mathbb{T}$. Also, the proof of Rogers~\cite{Rog57A} can be adapted to show the constant 1 in front of $n\ln n$ in Theorem~\ref{THM:poisson-process-lower-bound-cube} is best possible. Thus, rather surprisingly, an $n$-dimensional cube is a ``worst'' convex body for covering by random translations.

In the next section, we present the proof of Theorem~\ref{THM:poisson-process-upper-bound}.
In Section~\ref{SEC:lower-bound}, we give the proofs of Theorems~\ref{THM:poisson-process-lower-bound} and~\ref{THM:poisson-process-lower-bound-cube}.
%We note that our proof of Theorem~\ref{THM:poisson-process-upper-bound} is largely inspired by the work of Dumer~\cite{07D}.
%, in which he improved the classical Rogers bound on sphere coverings from $(1+o(1)) n \ln n$ to $(1/2 + o(1)) n \ln n$. 
%particularly his definition of saturated points (which appears in the proof below). 

%%%%%%%%%%%%%%%%%%%%%%%%%%%%%%%
\section{Proof of Theorem~\ref{THM:poisson-process-upper-bound}}\label{SEC:proof-upper-bound}
In this section, we present the proof of Theorem~\ref{THM:poisson-process-upper-bound}. 

We will use the following standard estimate for the tail of the Poisson distribution, which can be found, for example, in~{\cite[Theorem~5.4]{ME05}}.
\begin{lemma}\label{LEMMA:Poisson-concentration}
    Suppose that $\mathbf{X} \sim \mathrm{Pois}(\lambda)$. 
    Then the following statements hold. 
    \begin{enumerate}[label=(\roman*)]
        \item\label{LEMMA:Poisson-concentration-a} For every $\sigma > 0$, 
        \begin{align*}
            \mathbb{P}\big[\, \mathbf{X} \ge (1+\sigma) \lambda\, \big] 
            \le \me^{-\lambda} \left(\frac{\me \lambda}{(1+\sigma) \lambda}\right)^{(1+\sigma) \lambda}
            = \left(\frac{\me^{\sigma}}{(1+\sigma)^{1+\sigma}}\right)^{\lambda}. 
        \end{align*}
        \item\label{LEMMA:Poisson-concentration-b} For every $\sigma \in (0,1)$, 
        \begin{align*}
            \mathbb{P}\big[\, \mathbf{X} \le (1-\sigma) \lambda\, \big] 
            \le \left(\frac{\me^{-\sigma}}{(1-\sigma)^{1-\sigma}}\right)^{\lambda}. 
        \end{align*}
    \end{enumerate}
\end{lemma}

We will also use the following simple geometric result.
\begin{lemma}\label{LEMMA:volume-intersection-sphere}
    Let $\delta,r > 0$ be real numbers and let $n\to\infty$. 
    Suppose that $\mathbf{x}_{1}, \mathbf{x}_{2} \in \mathbb{R}^{n}$ are two points satisfying $\|\mathbf{x}_{1}-\mathbf{x}_{2}\| \le r n^{-{1}/{2} - \delta}$. Then
    \begin{align*}
        % \vol\big( B_{r}^{n}(\mathbf{x}_{1}) \cap B_{r}^{n}(\mathbf{x}_{2}) \big)
        % \ge \left(1 - \frac{1+o(1)}{\sqrt{2 \pi}} n^{-\delta} \right) \vol\big(B_{r}^{n}\big).
        \vol\big( B_{r}^{n}(\mathbf{x}_{1}) \setminus B_{r}^{n}(\mathbf{x}_{2}) \big)
        \le \frac{1+o(1)}{\sqrt{2 \pi}}\, n^{-\delta} \cdot \vol\big(B_{r}^{n}\big). 
    \end{align*}
\end{lemma}
\begin{proof}[Proof of Lemma~\ref{LEMMA:volume-intersection-sphere}]
    Let $\mathbf{x}_{1}, \mathbf{x}_{2} \in \mathbb{R}^{n}$ be two points such that $\|\mathbf{x}_{1}-\mathbf{x}_{2}\| \le r n^{-{1}/{2} - \delta}$.
    Using the facts 
    \begin{align*}
        \vol(B_{r}^{n})
        = \frac{\pi^{n/2}}{\Gamma\left(\frac{n}{2} + 1\right)}\, r^n
        \quad\text{and}\quad 
        \frac{\Gamma\left(\frac n2+1\right)}{\Gamma\left(\frac {n-1}2+1\right)} 
        =(1+o(1)) \sqrt{\frac n2}, 
    \end{align*} 
    we obtain  
    \begin{equation}\label{eq:temp1}
        \frac{\vol\big(B_{r}^{n-1}\big)}{\vol\big(B_{r}^{n}\big)}
         = \frac{1}{r \sqrt{\pi}} \frac{\Gamma\left(\frac{n}{2}+1\right)}{\Gamma\left(\frac{n-1}{2}+1\right)}
        = \frac{1+o(1)}{r \sqrt{2 \pi}} \sqrt{n}. 
    \end{equation}
    Let $\mathrm{conv}(X)$ denote the convex hull of a set $X\subseteq \mathbb{R}^n$. It follows from~\eqref{eq:temp1} that 
    \begin{align*}
        \vol\big( B_{r}^{n}(\mathbf{x}_{1}) \setminus B_{r}^{n}(\mathbf{x}_{2}) \big)
        & = \vol\big( B_{r}^{n}(\mathbf{x}_{1}) \cup B_{r}^{n}(\mathbf{x}_{2}) \big) - \vol\big( B_{r}^{n}(\mathbf{x}_{2}) \big) \\[0.3em]
        & \le \vol\big( \mathrm{conv}\big(B_{r}^{n}(\mathbf{x}_{1}) \cup B_{r}^{n}(\mathbf{x}_{2}) \big) \big) - \vol\big( B_{r}^{n}(\mathbf{x}_{2}) \big) \\[0.3em]
        & = %\|\mathbf{x}_{1} - \mathbf{x}_{2}\| 
        \|\mathbf{x}_{1}-\mathbf{x}_{2}\|\cdot \vol\big(B_{r}^{n-1}\big) \\[0.3em]
        & \le \frac{r}{ n^{1/2 + \delta}}\, \frac{1+o(1)}{r \sqrt{2 \pi}}\, \sqrt{n} \cdot \vol\big(B_{r}^{n}\big)
        = \frac{1+o(1)}{\sqrt{2 \pi}}\, n^{-\delta} \cdot \vol\big(B_{r}^{n}\big),  
    \end{align*}
    as desired. 
\end{proof}

We are now ready to present the proof of Theorem~\ref{THM:poisson-process-upper-bound}. 
\hide{
Here we work inside a packing torus $\mathbb{T}=\mathbb{R}^n/\Lambda$ of $B_2^n=B_1^n-B_1^n$. 
Note that the projection $\mathbb{R}^n\to \mathbb{T}$ is injective on any ball $B_1^n(\mathbf{p})$ in $\mathbb{R}^n$. Indeed, if we had $\mathbf{k}+\mathbf{p}+\mathbf{x}=\mathbf{k'}+\mathbf{p}$ for some  $\mathbf{x}\in\Lambda$ and distinct $\mathbf{k},\mathbf{k'}\in B_1^n$ then %$\mathbf{v}$ would be non-zero and 
$B_1^n\cap (B_1^n+\mathbf{x})$ would contain $\mathbf{k'}=\mathbf{k}+\mathbf{x}$, a contradiction
to our assumption that $\Lambda$ is a packing lattice for $B_1^n-B_1^n\supseteq B_1^n$.
In particular, the measure of a subset of any radius-1 ball, regardless whether we compute it inside $\mathbb{R}^n$ or inside~$\mathbb{T}$, justifying our usage of the same notation in both cases. Also, the distance between $\mathbf x,\mathbf y\in\mathbb{T}$ is defined as the distance between their preimage sets in $\mathbb{R}^n$ (and is also denoted by $\|\mathbf x-\mathbf y\|$). Note that the projection is a local isometry on each radius-1 ball in $\mathbb{R}^n$. We will use these facts a number of times in the proof of Theorem~\ref{THM:poisson-process-upper-bound} (as well as Theorems~\ref{THM:poisson-process-lower-bound} and~\ref{THM:poisson-process-lower-bound-cube}) without explicitly mentioning them.
}%
In this result as well as in 
Theorems~\ref{THM:poisson-process-lower-bound} and~\ref{THM:poisson-process-lower-bound-cube}, we work inside a packing torus $\mathbb{T}=\mathbb{R}^n/\Lambda$ of $K-K$ (where $K=B_1^n$ in Theorem~\ref{THM:poisson-process-upper-bound}). 
Note that the projection $\mathbb{R}^n\to \mathbb{T}$ is injective on any translate $K+\mathbf{p}\subseteq \mathbb{R}^n$. Indeed, if we had $\mathbf{k}+\mathbf{p}+\mathbf{x}=\mathbf{k'}+\mathbf{p}$ for some  $\mathbf{x}\in\Lambda$ and distinct $\mathbf{k},\mathbf{k'}\in K$ then %$\mathbf{v}$ would be non-zero and 
$K\cap (K+\mathbf{x})$ would contain $\mathbf{k'}=\mathbf{k}+\mathbf{x}$, a contradiction
to our assumption that $\Lambda$ is a packing lattice for $K-K$ (which contains a copy of $K$).
In particular, the measure of a subset of any translate of $K$ is the same, regardless whether we compute it inside $\mathbb{R}^n$ or inside~$\mathbb{T}$, justifying our usage of the same notation $\vol(\cdot)$ in both cases. Also, the distance between $\mathbf x,\mathbf y\in\mathbb{T}$ is defined as the distance between their preimage sets in $\mathbb{R}^n$ (and still denoted by $\|\mathbf x-\mathbf y\|$). Note that the projection is also an isometry on each radius-1 ball in $\mathbb{R}^n$. We will use these facts a number of times in the forthcoming proofs without explicitly mentioning them.

\begin{proof}[Proof of Theorem~\ref{THM:poisson-process-upper-bound}]
    Fix $\delta > 0$, which we may assume to be small. 
    Let $\beta = \beta(\delta) \in (0, 1)$ be a real number satisfying  
    \begin{align}\label{equ:beta-def}
        \left(\frac{\me^{\beta-1}}{\beta^\beta}\right)^{1+\frac{\delta}{2}} 
        \le \me^{-1}.
    \end{align}
    Since $\lim_{x \to 0} \frac{\me^{x-1}}{x^x} = \me^{-1}$ with convergence from above, such a real $\beta \in (0,1)$ exists. 
    
    Let $n$ be a sufficiently large integer.
    Define
    \begin{align*}
        \varepsilon \coloneqq \frac{1}{n \ln n}
        \quad\text{and}\quad 
        \mu \coloneqq n^{-\frac{1}{2} - \frac{\delta}{4}}. 
    \end{align*}

Let $\mathbb{T}$ be a packing torus of $B_{2}^{n}$ with volume at most $M^n \nu_{n}$. 
    Let $P_{\varepsilon} \subseteq \mathbb{T}$ and $P_{\mu} \subseteq \mathbb{T}$ be a maximal $B_{\varepsilon/2}^{n}$-packing and a maximal $B_{\mu/2}^{n}$-packing of $\mathbb{T}$, respectively. 
    By maximality, $P_{\varepsilon}$ forms a $B_{\varepsilon}$-covering of $\mathbb{T}$ and $P_{\mu}$ forms a $B_{\mu}$-covering of $\mathbb{T}$. 
    %    It follows from the maximality that $P_{\varepsilon}$ and $P_{\mu}$ are $B_{\varepsilon}^{n}$-covering and $B_{\mu}^{n}$-covering of $\mathbb{T}$, respectively. 
    Also, we have the trivial upper bounds
    \begin{equation}\label{equ:max-packing-upper-bound}
        |P_{\varepsilon}|
        \le \frac{\vol(\mathbb{T})}{\vol\big( B_{\varepsilon/2}^n \big)}
        \le \frac{M^n \nu_{n}}{(\varepsilon/2)^n \nu_{n}}
        =M^n \left(\frac{\varepsilon}{2}\right)^{-n}
        \quad\text{and}\quad 
        |P_{\mu}|
        \le M^n \left(\frac{\mu}{2}\right)^{-n}.
    \end{equation}
    Fix a map $\phi \colon P_{\varepsilon} \to P_{\mu}$ such that, for every $\mathbf{x} \in P_{\varepsilon}$, 
    \begin{align*}
        \|\mathbf{x}-\phi(\mathbf{x})\| \le \mu. 
    \end{align*}
    Note that such a map exists since $P_{\mu}$ is a $B_{\mu}^{n}$-covering. 

    Let $\mathbf{X}$ be a set sampled from the Poisson point process on $\mathbb{T}$ with intensity 
    \begin{align}\label{equ:rho-def-a}
        \rho \coloneqq \left(\frac{1}{2} + \delta\right) \frac{n \ln n}{\nu_{n}}. 
    \end{align}

    \begin{claim}\label{CLAIM:X-concentration}
        The probability that $|\mathbf{X}| \le (1+\delta) \rho \cdot \vol(\mathbb{T})$ tends to $1$ as $n \to \infty$. 
    \end{claim}
    \begin{proof}[Proof of Claim~\ref{CLAIM:X-concentration}]
        Since $\mathbb{T}$ is a packing torus of $B_{2}^{n}$, we trivially have $\vol(\mathbb{T}) \ge \vol(B_{2}^{n}) = 2^{n} \nu_{n}$. 
        The random variable $|\mathbf{X}|$ has the Poisson distribution with mean $\rho \cdot \vol(\mathbb{T}) \ge 2^{n-1} n \ln n$. 
        It follows from Lemma~\ref{LEMMA:Poisson-concentration}~\ref{LEMMA:Poisson-concentration-a} and the crude inequality $\frac{\me^{x}}{(1+x)^{1+x}} \le \me^{-x^2/10}$ for $x\in [0,1]$ that 
        \begin{align*}
            \mathbb{P}\left[\,|\mathbf{X}| \ge (1+\delta) \rho \cdot \vol(\mathbb{T})\,\right]
             \le \left(\frac{\me^{\delta}}{(1+\delta)^{1+\delta}}\right)^{\rho \cdot \vol(\mathbb{T})} 
             \le \mathrm{exp}\left(-\frac{\delta^2}{10} \cdot 2^{n-1} n \ln n\right),   
        \end{align*}
        which goes to $0$ as $n \to \infty$. 
    \end{proof}%CLAIM
    
    For every point $\mathbf{y} \in \mathbb{T}$ and $d \ge 0$, let 
    \begin{align*}
        N_{d}(\mathbf{y}) 
        \coloneqq B_{d}^{n}(\mathbf{y}) \cap \mathbf{X}
        = \left\{ \mathbf{x} \in \mathbf{X} \colon \|\mathbf{x}-\mathbf{y}\| \le d \right\}. 
    \end{align*}
    We say that a point $\mathbf{y}\in\mathbb{T}$ is \emph{saturated} if 
    \begin{align*}
        |N_{1-\varepsilon}(\mathbf{y})| \ge \frac{\beta}{2}\, n \ln n. 
    \end{align*}
    Define the events $E_{1}$, $E_{2}$ and $E_{3}$ as follows: 
    \begin{itemize}
        \item $E_{1}$: every point in $P_{\mu}$ is saturated.
        \item $E_{2}$: $|N_{1+\varepsilon}(\mathbf{z})| \le \left(\xi + 10\delta\right) n \ln n$ for every $\mathbf{z} \in P_{\varepsilon}$, where $\xi = 1.79556...$ is the unique real root of~\eqref{equ:xi-def}.  
        \item $E_{3}$: for each $\mathbf{z} \in P_{\varepsilon}$, either $|N_{1-\varepsilon}(\mathbf{z})|>0$, or $\phi(\mathbf{z})$ is not saturated.
    \end{itemize}
    % Let $E_{1}$ denote the event that 
    % \begin{align*}
    %     \text{every point in $P_{\mu}$ is saturated.}
    % \end{align*}
    % %
    % Let $E_{2}$ denote the event that 
    % \begin{align*}
    %     \text{$|N_{1+\varepsilon}(\mathbf{z})| \le \left(\xi + 10\delta\right) n \ln n$ for every $\mathbf{z} \in P_{\varepsilon}$}. 
    % \end{align*}
    % %
    % Let $E_{2}$ denote the event that 
    % \begin{align*}
    %     \text{for each $\mathbf{z} \in P_{\varepsilon}$ \quad either \quad $\mathrm{dist}(\mathbf{z}, \mathbf{X}) \le 1-\varepsilon$ \quad or \quad  $\phi(\mathbf{z})$ is not saturated.}
    % \end{align*}

Next, we show that each of these events holds with high probability.    

\begin{claim}\label{CLAIM:E1-E2-E3}
    We have $\min\left\{\,\mathbb{P}[E_1],\,\mathbb{P}[E_2],\,\mathbb{P}[E_3]\,\right\} \to 1$ as $n\to \infty$. 
\end{claim}
\begin{proof}[Proof of Claim~\ref{CLAIM:E1-E2-E3}]
    First, we prove that $\mathbb{P}(E_1) \to 1$. 
    Fix an arbitrary point $\mathbf{y} \in P_{\mu}$. 
    The random variable $|N_{1-\varepsilon}(\mathbf{y})|$ has the Poisson distribution with mean 
    \begin{align*}
        \rho \cdot \vol(B_{1-\varepsilon}^{n})
        = \left(\frac{1}{2} + \delta\right) \frac{n \ln n}{\nu_{n}} \cdot \vol(B_{1-\varepsilon}^{n})
        = \left( \frac{1}{2} + \delta\right) (1-\varepsilon)^{n} n \ln n.  
    \end{align*}
    Recall that $\varepsilon=(n\ln n)^{-1}$. Since $0<\delta <1$ is fixed while $n$ is sufficiently large, it holds that 
    \begin{align*}
        \left(\frac{1}{2}+\delta\right)(1-\varepsilon)^{n}
        \ge \frac{1}{2}
        \quad\text{and}\quad 
        \frac{2}{2+\delta} \left(\frac{1}{2}+\delta\right)(1-\varepsilon)^{n}
        \ge \frac{1}{2} + \frac{\delta}{3}. 
    \end{align*}
    Applying Lemma~\ref{LEMMA:Poisson-concentration}~\ref{LEMMA:Poisson-concentration-b} with $\sigma = 1 - \beta$, and using~\eqref{equ:beta-def}, we obtain 
    \begin{align*}
        \mathbb{P}\left[\,|N_{1-\varepsilon}(\mathbf{y})| \le \frac{\beta}{2}\, n \ln n\,\right]
        & \le \mathbb{P}\left[\,|N_{1-\varepsilon}(\mathbf{y})| \le \left(1 - (1-\beta) \right)\left( \frac{1}{2} + \delta\right) (1-\varepsilon)^{n} n \ln n\, \right] \\
        & \le \left(\frac{\me^{\beta-1}}{\beta^{\beta}}\right)^{\left( \frac{1}{2} + \delta\right) (1-\varepsilon)^{n} n \ln n} \\
        & \le \me^{-\frac{2}{2+\delta} \cdot \left( \frac{1}{2} + \delta\right) (1-\varepsilon)^{n} n \ln n} 
        \le \me^{-\left( \frac{1}{2} + \frac{\delta}{3}\right) n \ln n}
        = n^{-\left( \frac{1}{2} + \frac{\delta}{3}\right) n}. 
    \end{align*}
    Combining this with~\eqref{equ:max-packing-upper-bound}, we obtain 
    \begin{align*}
        \mathbb{P}[E_1]
        & \ge 1 - \sum_{\mathbf{y} \in P_{\mu}} \mathbb{P}\left[\,|N_{1-\varepsilon}(\mathbf{y})| \le \frac{\beta}{2}\, n \ln n\,\right] \\
        & \ge 1 - M^{n} \left(\frac{\mu}{2}\right)^{-n} n^{-\left( \frac{1}{2} + \frac{\delta}{3}\right) n}
        % = 1 - \left(\frac{n^{\frac{1}{2} + \frac{\delta}{3} - \left(\frac{1}{2} + \frac{\delta}{4}\right)}}{8}\right)^{-n}
        = 1 - \left(\frac{n^{\delta/12}}{2M}\right)^{-n}
        \to 1 \quad\text{as}\quad n \to \infty, 
    \end{align*}
    as desired. 

    We now consider $E_2$. 
    Fix an arbitrary point $\mathbf{z} \in P_{\varepsilon}$. 
    As in the preceding proof, the random variable $|N_{1+\varepsilon}(\mathbf{z})|$ has the Poisson distribution with mean 
    \begin{align*}
        \rho \cdot \vol(B_{1+\varepsilon}^{n})
        = \left(\frac{1}{2} + \delta\right) \frac{n \ln n}{\nu_{n}} \cdot \vol(B_{1+\varepsilon}^{n})
        = \left( \frac{1}{2} + \delta\right) (1+\varepsilon)^{n} n \ln n.  
    \end{align*}
    Let $\xi_0 \coloneqq 2\xi - 1 = 2.59112...\, $. 
    Since $\xi$ is the root of~\eqref{equ:xi-def}, a straightforward calculation shows that $\xi_{0}$ is the unique real root of 
    \begin{align}\label{equ:xi0-def}
        \frac{\me^{x}}{(1+x)^{1+x}} = \frac{1}{\me^2}. 
    \end{align}
    %
    %Since 
    %$\lim_{n\to \infty} (1+\varepsilon)^n = 1$, we may choose 
    %$n$ is sufficiently large, 
    Also, we have
    \begin{align*}
        \left(\frac{1}{2} + \delta\right)(1+\varepsilon)^{n}
        \le \frac{1}{2} + 2 \delta. 
    \end{align*}
    Since $\xi = 1.79...$ is at most $5/2$, we obtain 
    \begin{align*}
        \frac{\xi + 10 \delta}{(1/2+\delta) (1+\varepsilon)^{n}} - (1+\xi_0) 
        \ge \frac{\xi + 10 \delta}{1/2 + 2\delta} -  2\xi 
        = \frac{4(5-2\xi)\delta}{1+4\delta} 
        \ge 0. 
    \end{align*}
    Now applying Lemma~\ref{LEMMA:Poisson-concentration}~\ref{LEMMA:Poisson-concentration-a} with $\sigma = \xi_0$, and using~\eqref{equ:xi0-def}, we obtain 
    \begin{align*}
        \mathbb{P}\left[\,|N_{1+\varepsilon}(\mathbf{z})| \ge (\xi + 10\delta) n \ln n\,\right]
        % & \le \mathbb{P}\left[ |N_{1+\varepsilon}(\mathbf{y})| \ge \frac{\xi + 10 \delta}{(1/2 + \delta)(1+\varepsilon)^{n}} \left( \frac{1}{2} + \delta\right) (1+\varepsilon)^{n} n \ln n \right]
        & \le \mathbb{P}\left[\, |N_{1+\varepsilon}(\mathbf{z})| \ge (1 + \xi_0) \left( \frac{1}{2} + \delta\right) (1+\varepsilon)^{n} n \ln n\, \right] \\
        & \le \left(\frac{\me^{\xi_0}}{(1+\xi_0)^{1+\xi_0}}\right)^{\left( \frac{1}{2} + \delta\right) (1+\varepsilon)^{n} n \ln n} \\
        & = \me^{-2 \left( \frac{1}{2} + \delta\right) (1+\varepsilon)^{n} n \ln n}
        \le \me^{-(1+2\delta) n \ln n}
        = n^{-(1+2\delta)n}. 
    \end{align*}
    Combining it with~\eqref{equ:max-packing-upper-bound}, we obtain 
    \begin{align*}
        \mathbb{P}[E_2]
        & \ge 1 - \sum_{\mathbf{z} \in P_{\varepsilon}} \mathbb{P}\left[\,|N_{1+\varepsilon}(\mathbf{z})| \ge (\xi + 10\delta) n \ln n\,\right] \\
        & \ge 1 - M^n \left(\frac{\varepsilon}{2}\right)^{-n} n^{-(1+2\delta)n} \\
        & = 1 - \left(\frac{n^{1+2\delta}}{2M n \ln n}\right)^{-n} 
        = 1 - \left(\frac{n^{2\delta}}{2M \ln n}\right)^{-n}
        \to 1 
        \quad\text{as}\quad n \to \infty, 
    \end{align*}
    as desired. 

    Finally, we consider $E_3$. 
    For a point  $\mathbf{z} \in P_{\varepsilon}$, let $B_{\mathbf{z}}$ denote the (bad) event that 
    \begin{align*}
        \text{$|N_{1-\varepsilon}(\mathbf{z})|=0$ \quad and\quad $\phi(\mathbf{z}) \in P_{\mu}$ is saturated}. 
    \end{align*}
    Next, we analyze the probability that $B_{\mathbf{z}}$ occurs. 

    Let $S \coloneqq B_{1-\varepsilon}^{n}\big(\phi(\mathbf{z})\big) \setminus B_{1-\varepsilon}^{n}(\mathbf{z})$. 
    By the definition of $\phi$, we have 
    \begin{align*}
        \|\mathbf{z}-\phi(\mathbf{z})\| 
        \le \mu 
        = n^{-\frac{1}{2} - \frac{\delta}{4}}
        \le (1-\varepsilon) n^{-\frac{1}{2} - \frac{\delta}{8}}. 
    \end{align*}
    So it follows from Lemma~\ref{LEMMA:volume-intersection-sphere} that 
    \begin{align}\label{equ:vol-S}
        \vol(S) 
        \le \frac{1+o(1)}{\sqrt{2 \pi}}\, n^{-\delta/8} \cdot \vol\big(B_{1-\varepsilon}^{n}\big)
        \le n^{-\delta/8} \nu_{n}. 
    \end{align}
    Since the random variable $|S \cap \mathbf{X}|$ has the Poisson distribution with mean 
    \begin{align*}
        \rho \cdot \vol(S) 
        = \left(\frac{1}{2}+\delta\right) \frac{n \ln n}{\nu_{n}}\, \vol(S)
        \le \left(\frac{1}{2}+\delta\right) n^{1-\frac{\delta}{8}} \ln n
        \ll \frac{\beta}{2}\, n \ln n, 
    \end{align*}
    it follows from Lemma~\ref{LEMMA:Poisson-concentration}~\ref{LEMMA:Poisson-concentration-a} and~\eqref{equ:rho-def-a} that 
    \begin{align*}
        \mathbb{P}\left[\,|S\cap \mathbf{X}| \ge \frac{\beta}{2}\, n \ln n\,\right]
        & \le \me^{- \rho \cdot \vol(S)} \left(\frac{\me \rho \cdot \vol(S)}{(\beta/2) n \ln n}\right)^{\frac{\beta}{2}\, n \ln n}\\ 
        & \le \left(\frac{\beta n \ln n}{2 \me \rho \cdot \vol(S)}\right)^{-\frac{\beta}{2}\, n \ln n}
        \le \left(\frac{\beta \nu_n }{2 \me \cdot \vol(S)}\right)^{-\frac{\beta}{2}\, n \ln n}. 
    \end{align*}
    Since $\beta = \beta(\delta)> 0$ while $n$ is sufficiently large, we have that 
    \begin{align*}
        \frac{\beta}{2\me} \ge n^{-\frac{\delta}{40}}. 
    \end{align*}
    Combining this with~\eqref{equ:vol-S}, we obtain 
    \begin{align*}
        \mathbb{P}\left[\,|S\cap \mathbf{X}| \ge \frac{\beta}{2}\, n \ln n\,\right]
        \le \left( n^{-\frac{\delta}{40}} \frac{\nu_{n}}{n^{-\delta/8} \nu_{n}}\right)^{-\frac{\beta}{2}\, n \ln n}
        = n^{-\frac{\beta \delta}{20}\, n \ln n}. 
    \end{align*}
    By definition, $B_{\mathbf{z}}$ occurs if and only if 
    \begin{align*}
        |N_{1-\varepsilon}(\mathbf{z})| = 0
        \quad\text{and}\quad 
        |N_{1-\varepsilon}(\phi(\mathbf{z}))| \ge \frac{\beta}{2}\, n \ln n, 
    \end{align*}
    which implies that all points in the set $N_{1-\varepsilon}(\phi(\mathbf{z}))$ lie within the region $S$. 
    Therefore, for any given $\mathbf{z}\in P_{\varepsilon}$, it holds that
    \begin{align*}
        \mathbb{P}[B_{\mathbf{z}}]
        \le \mathbb{P}\left[\,|S \cap \mathbf{X}| \ge \frac{\beta}{2} n \ln n\,\right]
        \le n^{-\frac{\beta \delta}{20} n \ln n}. 
    \end{align*}
    Combining it with~\eqref{equ:max-packing-upper-bound}, we obtain 
    \begin{align*}
        \mathbb{P}[E_3]
        \ge 1 - \sum_{\mathbf{z} \in P_{\varepsilon}} \mathbb{P}[B_{\mathbf{z}}]
        & \ge 1 - M^n \left(\frac{\varepsilon}{2}\right)^{-n} n^{-\frac{\beta \delta}{20}\, n \ln n} \\
        % & = 1 - \mathrm{exp}\left(n \ln\left(\frac{4}{e}\right) + n \ln\left(2 n \ln n\right) - \frac{\delta \beta}{20} n \ln^2 n \right)
        & = 1 - \mathrm{exp}\left( n \ln(2M\, n \ln n) - \frac{\beta \delta}{20} n \ln^2 n \right)
        \to 1 \quad\text{as}\quad n \to \infty. 
    \end{align*}
    This completes the proof of Claim~\ref{CLAIM:E1-E2-E3}. 
\end{proof}%CLAIM

Thus there is an outcome $\mathbf{X}$ of size at most $(1+\delta) \rho \cdot \vol(\mathbb{T})$ for which all of the events $E_1$, $E_2$, and $E_3$ occur. Fix any such $\mathbf{X}$.

\begin{claim}\label{CLAIM:X-is-covering}
    The set $\mathbf{X}$ is a unit ball covering of $\mathbb{T}$. Moreover, 
    \begin{align*}
        \vartheta(\mathbf{X}) \le \left(\frac{1}{2} + 2\delta\right) n \ln n
        \quad\text{and}\quad 
        \mu(\mathbf{X}) \le \left(\xi + 10 \delta\right) n \ln n. 
    \end{align*}
\end{claim}
\begin{proof}[Proof of Claim~\ref{CLAIM:X-is-covering}]
    First, we show that $\mathbf{X}$ is a unit ball covering of $\mathbb{T}$. 
    Fix an arbitrary point $\mathbf{w} \in \mathbb{T}$. 
    Since $P_{\varepsilon}$ is a $B_{\varepsilon}^{n}$-covering of $\mathbb{T}$, there exists a point $\mathbf{z} \in P_{\varepsilon}$ such that $\|\mathbf{w}- \mathbf{z}\| \le \varepsilon$. 
    The fact that the event $E_1$ holds gives that the point $\phi(\mathbf{z}) \in P_{\mu}$ is saturated. 
    Combining it with the fact that the event $E_3$ holds, we conclude that there exists a point $\mathbf{x} \in \mathbf{X}$ such that $\|\mathbf{z}- \mathbf{x}\| \le 1- \varepsilon$. 
    Therefore, 
    \begin{align*}
        \|\mathbf{w}- \mathbf{x}\|
        \le \|\mathbf{w}- \mathbf{z}\| + \|\mathbf{z}- \mathbf{x}\| 
        \le 1, 
    \end{align*}
    which shows that $\mathbf{X}$ is a unit ball covering of $\mathbb{T}$. 

    By definition, the covering density of $\mathbf{X}$ is 
    \begin{align*}
        \vartheta(\mathbf{X})
        = \sum_{\mathbf{x} \in \mathbf{X}} \frac{\vol\left(B_{1}^{n}(\mathbf{x}) \cap \mathbb{T}\right)}{\vol(\mathbb{T})}
        = |\mathbf{X}|\, \frac{\nu_{n}}{\vol(\mathbb{T})}
        & \le (1+\delta) \left(\frac{1}{2} + \delta\right) \frac{n \ln n}{\nu_{n}} \vol(\mathbb{T})\, \frac{\nu_{n}}{\vol(\mathbb{T})} \\
        & \le \left(\frac{1}{2} + 2\delta\right) n \ln n, 
    \end{align*}
    as desired. 

    Now suppose to the contrary that there exists a point $\mathbf{u} \in \mathbb{T}$ such that  
    \begin{align*}
        \mu(\mathbf{X}, \mathbf{u}) > (\xi + 10 \delta) n \ln n. 
    \end{align*}
    Since $P_{\varepsilon}$ is a $B_{\varepsilon}^{n}$ covering of $\mathbb{T}$, there exists a point $\mathbf{z} \in P_{\varepsilon}$ such that $\|\mathbf{u}- \mathbf{z}\| \le \varepsilon$. 
    Then 
    \begin{align*}
        |N_{1+\varepsilon}(\mathbf{z})|
        \ge \mu(\mathbf{X}, \mathbf{u}) > (\xi + 10 \delta) n \ln n, 
    \end{align*}
    which contradicts the fact that the event $E_2$ occurs. 
\end{proof}%CLAIM

Claim~\ref{CLAIM:X-is-covering} completes the proof of Theorem~\ref{THM:poisson-process-upper-bound}.     
\end{proof}

%%%%%%%%%%%%%%%%%%%%%%%
\section{Proofs of Theorems~\ref{THM:poisson-process-lower-bound} and~\ref{THM:poisson-process-lower-bound-cube}}\label{SEC:lower-bound}
In this section, we present the proof of Theorems~\ref{THM:poisson-process-lower-bound} and~\ref{THM:poisson-process-lower-bound-cube}. 

A convex body $K \subseteq \mathbb{R}^n$ is said to be \emph{isotropic} if $\int_{K} \mathbf{x}\, \mathrm{d}\mathbf{x} = 0$, $\vol(K) = 1$, and there exists a constant $L_K$, called the \emph{isotropic constant} of $K$, such that for every $\mathbf{y} \in \mathbb{R}^{n}$, 
\begin{align}\label{equ:def-isotropic}
    \int_K \langle \B x, \B y \rangle^2 \, \mathrm{d}\B x 
    = L_K^2 \|\B y\|^2.
    % \quad \text{for every}\quad \B x \in \mathbb{R}^n.
\end{align}
% where  $\|\B y\|=\mathrm{dist}(\B y,\B 0)$ denotes the Eucledian norm of $\B y$.
%

It was a long-standing open problem to determine whether $L_K$ can be bounded by an absolute constant, independent of the dimension $n$. 
Culminating decades of work by many researchers, the question was recently resolved affirmatively by Guan, Klartag and Lehec~\cite{guan24,KL24}.

\begin{theorem}[Guan~\cite{guan24}, Klartag and Lehec~\cite{KL24}]\label{THM:GKL24-isotropic-constant}
    There exists an absolute constant $C$ such that the following holds for every $n \in \mathbb{N}$.
    Suppose that $K \subseteq \mathbb{R}^n$ is an isotropic convex body. Then the isotropic constant $L_{K}$ of $K$ satisfies $L_{K} \le C$. 
\end{theorem}

Since every convex body can be made isotropic via a suitable linear transformation (see e.g.~%{\cite[Theorem~4.2]{Tkocz2018}}
\cite[Solution to Exercise 9.17]{BalestroMartiniTeixeira25cgpves}), it suffices to prove Theorem~\ref{THM:poisson-process-lower-bound} for isotropic convex bodies. 
The following lemma will be used in the proof. 
\begin{lemma}\label{LEMMA:small-overlap}
    Suppose that $K \subseteq \mathbb{R}^{n}$ is an isotropic convex body and $\mathbf{x} \in \mathbb{R}^{n}$ is a vector satisfying $\|\mathbf{x}\| \ge 4 L_{K}$. 
    Then 
    \begin{align*}
        \vol\left( K \cap (K + \mathbf{x}) \right)
        \le 3^{-\frac{\|\mathbf{x}\|}{8L_{K}}}. 
    \end{align*}
\end{lemma}
\begin{proof}[Proof of Lemma~\ref{LEMMA:small-overlap}]
    Fix  a vector $\mathbf{x} \in \mathbb{R}^{n}$ satisfying $\|\mathbf{x}\| \ge 4 L_{K}$. 
    For every real number $w \in \mathbb{R}$, define the \emph{slab of width $w$ in the direction of $\mathbf{x}$} as 
    \begin{align*}
        S_{\mathbf{x}}(w)
        \coloneqq \left\{\mathbf{y} \colon |\langle \B x, \B y \rangle| < w\, \|\mathbf{x}\| \right\}. 
    \end{align*}
    Markov’s inequality applied to the random variable $\langle \B x, \B y \rangle^2$, where $\B y$ is a uniform element of $K$, implies by~\eqref{equ:def-isotropic}  that 
    \begin{align}\label{eq:3/4}
        \vol\left(K \setminus S_{\mathbf{x}}(2L_{K})\right)
        \le \frac{1}{4}. 
    \end{align}

 A well-known argument due to Borell~{\cite[Lemma~3.1]{Borell74}} gives that 
    \begin{align}\label{eq:t}
        \vol\left(K \setminus S_{\mathbf{x}}(2tL_K)\right)
        \le 
        %\frac{3}{4} \cdot 3^{-\frac{t+1}{2}}= 
        \frac{\sqrt{3}}{4} \cdot 3^{-\frac{t}{2}},
        \quad\text{for every}\quad t \ge 1. 
    \end{align}
For completeness, let us  sketch a proof of~\eqref{eq:t}. The uniform probability measure $\mu_K$ on $K$ is clearly \emph{logarithmically concave}, meaning that the function $-\log(\dd \mu_K/\dd \mu)$ (which, in this case, is $0$ on $K$ and $\infty$ elsewhere) is convex. By a theorem of Borell~\cite[Theorem~3.2]{Borell75}, $\mu_K$ satisfies the Multiplicative Brunn-Minkowski Inequality. %for any  Borel subsets of~$\mathbb{R}^n$.
Next, it is routine to check that, for any convex centrally symmetric $A$, it holds that
   \[
   \mathbb{R}^n\setminus A\supseteq \frac{2}{t+1}\,(\mathbb{R}^n\setminus tA)+\frac{t-1}{t+1}\, A.
   \]
Evaluating $\mu_K$ on this inclusion for the slab $A\coloneqq S_{\mathbf{x}}(2L_K)$, we obtain that
   \[
   %\mu_k(\mathbb{R}^n\setminus A)=
   \vol(K\setminus A)\ge \mu_K\left(\frac{2}{t+1}\,(\mathbb{R}^n\setminus tA)+\frac{t-1}{t+1}\, A\right)\ge \vol(K\setminus tA)^{\frac2{t+1}}\cdot\vol(K\cap A)^{\frac{t-1}{t+1}},
   \]
   which implies~\eqref{eq:t} by~\eqref{eq:3/4}.

    In particular, by applying~\eqref{eq:t} with $t \coloneqq {\|\mathbf{x}\|}/(4L_{K})$ and denoting $S \coloneqq S_{\mathbf{x}}\left(\|\mathbf{x}\|/2\right)$ for brevity, we obtain 
    \begin{align*}
        \vol\left(K \setminus S\right)
        \le \frac{\sqrt{3}}{4} \cdot 3^{-\frac{\|\mathbf{x}\|}{8L_{K}}}. 
    \end{align*}
    \begin{claim}\label{CLAIM:S-Sx-empty-intersection}
        We have that $S \cap (S + \mathbf{x}) = \emptyset$. 
    \end{claim}
    \begin{proof}[Proof of Claim~\ref{CLAIM:S-Sx-empty-intersection}]
        Suppose to the contrary that there exists $\mathbf{y} \in S \cap (S + \mathbf{x})$. 
        Then $\mathbf{y} \in S$ and $\mathbf{y} - \mathbf{x} \in S$. 
        It follows from the definition of $S$ that 
        \begin{align*}
            \|\mathbf{x}\|^2
            = |\langle \mathbf{x}, \mathbf{x} \rangle|
            = |\langle \mathbf{y}, \mathbf{x} \rangle - \langle \mathbf{y} - \mathbf{x}, \mathbf{x} \rangle|
            \le |\langle \mathbf{y}, \mathbf{x} \rangle| + |\langle \mathbf{y} - \mathbf{x}, \mathbf{x} \rangle|
            < \frac{\|\mathbf{x}\|^2}{2} + \frac{\|\mathbf{x}\|^2}{2}
            = \|\mathbf{x}\|^2, 
        \end{align*}
        a contradiction. 
    \end{proof}%CLAIM
    Let $\overline{S} \coloneqq K \setminus S$. 
    It follows from Claim~\ref{CLAIM:S-Sx-empty-intersection} that 
    \begin{align*}
        K \cap (K + \mathbf{x})
        \subseteq \overline{S} \cup (\overline{S} + \mathbf{x}). 
    \end{align*}
    Therefore, 
    \begin{align*}
        \vol\left(K \cap (K + \mathbf{x})\right)
        \le \vol(\overline{S}) + \vol(\overline{S} + \mathbf{x})
        % = 2 \vol(\overline{S})
        = 2 \cdot \vol(K \setminus S)
        \le 2 \cdot \frac{\sqrt{3}}{4} \cdot 3^{-\frac{\|\mathbf{x}\|}{8L_{K}}}
        \le 3^{-\frac{\|\mathbf{x}\|}{8L_{K}}},  
    \end{align*}
    which proves Lemma~\ref{LEMMA:small-overlap}. 
\end{proof}

We will also use the following standard result from probability, which can be found in e.g.~{\cite[Theorem~4.3.1]{AS16}}. 
\begin{lemma}\label{LEMMA:second-moment}
    Suppose that $X$ is a nonnegative, integral-valued random variable. 
    Then 
    \begin{align*}
        \mathbb{P}[\,X = 0\,] \le \frac{\mathrm{Var}[X]}{\mathbb{E}[X]^2}, 
    \end{align*}
    where $\mathrm{Var}[X]$ and $\mathbb{E}[X]$ denote respectively the variance and the expectation of $X$.
\end{lemma}

Before proving Theorem~\ref{THM:poisson-process-lower-bound}, we first establish the following result.%\op{Rephrased a bit}
\begin{theorem}\label{THM:poisson-process-lower-bound-isotropic} Let $n\to\infty$.
    Let $K \subseteq \mathbb{R}^{n}$ be an isotropic convex body and let 
     $\mathbb{T}$ be a packing torus of $K-K$. Suppose that  $P \subseteq \mathbb{T}$ is a $B_{1/2}^{n}$-packing 
     and $\omega(n)\to \infty$ is a function such that
     %of size at least  $\lceil \me^{\me}\rceil=16$. 
 %    Let $\omega(n)\to \infty$ be an arbitrary function such that
    \begin{align*}
        \rho
        \coloneqq \ln |P| - n \big( \ln\ln\ln |P| + \omega(n) \big) 
    \end{align*}
    tends to infinity.  Let $\mathbf{X} \subseteq \mathbb{T}$ be a set sampled from the Poisson point process on $\mathbb{T}$ with intensity $\rho$.
    Then $\mathbb{P}\left[\,P \subseteq K + \mathbf{X}\,\right]$ tends to $0$. 
    % \begin{align*}
    %     \mathbb{P}\left[P \subseteq K + \mathbf{X}\right] \to 0 
    %     \quad\text{as}\quad n \to \infty. 
    % \end{align*}
\end{theorem}
\begin{proof}[Proof of Theorem~\ref{THM:poisson-process-lower-bound-isotropic}] 
%Here we work inside the torus~$\mathbb{T}$. 
For every point $\mathbf{p} \in P$, let $B_{\mathbf{p}}$ be the characteristic function of the event that $\B p \notin \mathbf{X} + K$, that is, 
    \begin{align*}
        B_{\mathbf{p}}
        = 
        \begin{cases}
            1, &\quad\text{if}\quad \B p \notin \mathbf{X} + K, \\
            0, &\quad\text{if}\quad \B p \in \mathbf{X} + K.
        \end{cases}
    \end{align*}
    Since $\vol(K) = 1$ (by the definition of isotropic), the random variable $|\mathbf{X} \cap (\mathbf{p} - K)|$ has the Poisson distribution with mean $\rho \cdot \vol(\mathbf{p} - K) = \rho$. 
    Therefore, $\mathbb{E}[B_{\mathbf{p}}] = \me^{-\rho}$.
    It follows that the sum $B \coloneqq \sum_{\mathbf{p} \in P}B_{\mathbf{p}}$ satisfies 
    \begin{align}\label{equ:E-B}
        \mathbb{E}[B] 
        = |P|\, \me^{-\rho}
        = \me^{n \left( \ln\ln\ln |P| + \omega(n) \right)}
        \to \infty
        \quad\text{as}\quad 
        n \to \infty. 
    \end{align}
    Note that $B$ is the number of points in $P$ that are not covered by $\mathbf{X} + K$.  
    We will use Lemma~\ref{LEMMA:second-moment} (that is, the second moment method) to show that $\PP[\,B = 0\,]\to 0$ as $n \to \infty$.

    Let $\Delta \coloneqq 8 L_K (f_n + \log_3 \rho)$, where $f_n \to \infty$ very slowly (as $n \to \infty$). 
    Let 
    \begin{align*}
        \Sigma_1 \coloneqq 
        \sum_{\substack{\B p,\B p'\in P\\ \|\B p - \B p'\| \leq \Delta}} \mathbb{E}[B_\B p B_{\B p'}]
        \quad\text{and}\quad 
        \Sigma_2 \coloneqq 
        \sum_{\substack{\B p,\B p'\in P\\ \|\B p - \B p'\| > \Delta}} \left( \mathbb{E}[B_\B p B_{\B p'}] - \mathbb{E}[B_\B p]\, \mathbb{E}[B_{\B p'}] \right). 
    \end{align*}
    Since $P$ is a $B_{1/2}^{n}$-packing of $\mathbb{T}$, for every point $\mathbf{p} \in P$ the number of points in $P$ satisfying $\|\B p - \B p'\| \leq \Delta$ is at most 
    \begin{align*}
        \frac{\vol\big(B_{\Delta+1/2}^{n}(\mathbf{p})\big)}{\vol\big(B_{1/2}^{n}\big)}
        = \left(\frac{\Delta + 1/2}{1/2}\right)^{n}
        \le \left( 4\Delta \right)^{n}. 
    \end{align*}
    Notice that 
    \begin{align*}
         \ln&\left(\frac{\left( 32 L_K (f_n + \log_3 \rho) \right)^{n}}{|P| \me^{-\rho}}\right) \\
        & = n \ln\left(32 L_K (f_n + \log_3 \rho)\right) - n \big( \ln\ln\ln |P| + \omega(n) \big) \\
        & = n \ln\left( f_n + \frac{\ln \rho}{\ln 3}\right) + n \ln (32 L_K) - n \big( \ln\ln\ln |P| + \omega(n) \big) \\
        & \le n \ln\ln \rho  + n \ln (32 L_K) - n \big( \ln\ln\ln |P| + \omega(n) \big) 
        \le n \big( \ln (32 L_K) - \omega(n) \big)
    \end{align*}
    which, by Theorem~\ref{THM:GKL24-isotropic-constant}, tends to $-\infty$ as $n \to \infty$. 
    Therefore, by~\eqref{equ:E-B}, 
    \begin{align}\label{equ:Sigma-1-upper-bound}
        \Sigma_1
        \le \sum_{\mathbf{p} \in P} \mathbb{E}[B_{\mathbf{p}}] \left(4\Delta\right)^{n} 
        = \mathbb{E}[B] \left(4\Delta\right)^{n}
        = \mathbb{E}[B]^2 \frac{\left( 32 L_K (f_n + \log_3 \rho) \right)^{n}}{|P|\, \me^{-\rho}} 
        = o\left(\mathbb{E}[B]^{2}\right).
    \end{align}

    Next, we consider $\Sigma_2$. 
    Fix two points $\mathbf{p},\mathbf{p}' \in P$ at distance at least $\Delta$ on the torus $\mathbb{T}$. 
    Note that $B_{\mathbf{p}} B_{\mathbf{p}'}$ is 1 if and only if $\mathbf{X}$ does not intersect $(\mathbf{p} - K) \cup (\mathbf{p}' - K)$.
\hide{
We claim that the quotient map $\mathbb{R}^n\to \mathbb{T}$ is injective on $(K - \mathbf{p}) \cup (K - \mathbf{p}')$, which actually implies that the volume of $(K - \mathbf{p}) \cup (K - \mathbf{p}')$ on $\mathbb{R}^n$ is preserved under the quotient map $\mathbb{R}^n\to \mathbb{T}$. \op{What do we mean by $(K - \mathbf{p}) \cup (K - \mathbf{p}')$ inside $R^n$? $p,p'$ are in torus so different lifts of them to $R^n$ can give different non-equivalent intersections when projected back to the torus (basically $K$ can be very thin and long and its intersection with a fundamental domain can be made of say 100, or even $\omega(n)$, different translates)}
    %\op{Claim instead that the projection is injective on each union $(K - \mathbf{p}) \cup (K - \mathbf{p}')$? This is more natural to me and implies the volume preservation claim.}
    First, suppose that $(K-\mathbf{p})\cap (K - \mathbf{p}')=\emptyset$ on $\mathbb{T}$, then $(K-\mathbf{p})\cap (K - \mathbf{p}')=\emptyset$ on $\mathbb{R}^n$. Since $K\subseteq K-K$ and $\mathbb{T}$ is a packing torus of $K-K$, the quotient map is injective on $(K - \mathbf{p}) \cup (K - \mathbf{p}')$. The other case is that $K-\mathbf{p}$ and $K - \mathbf{p}'$ intersects on some point $\mathbf{x}\in \mathbb{T}$. Then the conclusion follows from the fact that $(K - \mathbf{p}) \cup (K - \mathbf{p}')\subseteq K-K+\mathbf{x}$ and $\mathbb{T}$ is a packing torus of $K-K$. 
}     
Hence, the number of points of $\mathbf{X}$ in this set has the Poisson distribution with mean 
    \begin{align*}
        \rho \cdot \vol\left( (\mathbf{p} - K) \cup (\mathbf{p}' - K)\right)
        & = \rho \cdot \vol\left( (K - \mathbf{p}) \cup (K - \mathbf{p}')\right) \\
        & = \rho \left(\vol(K - \mathbf{p}) + \vol(K - \mathbf{p}') - \vol\left( (K - \mathbf{p}) \cap (K - \mathbf{p}')\right)\right) \\
        & = 2\rho - \rho \cdot \vol\left(K \cap (K + \mathbf{p} - \mathbf{p}')\right).  
    \end{align*}
    Thus we have that
    \begin{align}\label{equ:exp-Bp-Bp'}
        \mathbb{E}[B_{\mathbf{p}} B_{\mathbf{p}'}]
        & = \mathrm{exp}\left(-\rho \cdot \vol\left( (\mathbf{p} - K) \cup (\mathbf{p}' - K)\right)\right) \notag \\
        & = \mathrm{exp}\left( - 2\rho + \rho \cdot \vol\left(K \cap (K + \mathbf{p} - \mathbf{p}')\right) \right). 
    \end{align}
Observe that there is at most one vector $\mathbf{q}\in\mathbb R^n$ that projects to $\mathbf{p}-\mathbf{p}'$ with $K\cap (K+\mathbf{q})\not=\emptyset$. Indeed, if $\mathbf{q}\not=\mathbf{q}'$ contradict this, satisfying $\mathbf{a}=\mathbf{b}+\mathbf{q}$ and $\mathbf{a}'=\mathbf{b}'+\mathbf{q}'$ for some $\mathbf{a},\mathbf{b},\mathbf{a}',\mathbf{b}'$ in $K\subseteq \mathbb{R}^n$ 
then 
$$
(\mathbf{a}-\mathbf{b})-(\mathbf{a}'-\mathbf{b}')=\mathbf{q}-\mathbf{q}'\in\Lambda\setminus\{\mathbf{0}\},
$$ 
a contradiction to our assumption that $\Lambda$ is a packing lattice for $K-K$. Furthermore, for every $\mathbf{q}\in\mathbb R^n$ that projects to $\mathbf{p}-\mathbf{p}'$, it holds that $\|\mathbf{q}\| 
%\ge \|\mathbf{p}-\mathbf{p}'\| 
\ge \Delta \ge 4 L_{K}$. 
Thus the upper bound of Lemma~\ref{LEMMA:small-overlap} applies also inside the torus and gives that
    \begin{align*}
        \vol\left(K \cap (K + \mathbf{p} - \mathbf{p}')\right)
        \le 3^{-\frac{\|\mathbf{p} - \mathbf{p}'\|}{8L_{K}}}
        \le 3^{-\frac{\Delta}{8L_{K}}}
        = 3^{- f_n - \log_{3} \rho}
        = \rho^{-1} 3^{-f_n}. 
    \end{align*}
    Combining it with~\eqref{equ:exp-Bp-Bp'}, we obtain  
    \begin{align}\label{equ:Sigma-2-upper-bound}
        \Sigma_2 
        & \le |P|^{2} \left(\mathrm{exp}\left(-2\rho + 3^{-f_n}\right) - \mathrm{exp}(-2\rho) \right) \notag \\
        & = |P|^{2} \me^{-2\rho} \left(\mathrm{exp}\left(3^{-f_n}\right) - 1\right)
        % |P|^{2} \me^{-2\rho} \left(\mathrm{exp}\left(\rho 3^{-\frac{\Delta}{8L_{K}}}\right) - 1 \right) 
         = \mathbb{E}[B]^{2} \left(\mathrm{exp}\left(3^{-f_n}\right) - 1\right)
        = o\left(\mathbb{E}[B]^{2}\right). 
    \end{align}
    Combining~\eqref{equ:Sigma-1-upper-bound},~\eqref{equ:Sigma-2-upper-bound}, and Lemma~\ref{LEMMA:second-moment}, we obtain 
    \begin{align*}
        \mathbb{P}[\,B = 0\,]
        \le \frac{\mathrm{Var}[B]}{\mathbb{E}[B]^{2}}
        = \frac{1}{\mathbb{E}[B]^{2}} \sum_{\B p, \B p' \in P} \left( \mathbb{E}[B_\B p B_{\B p'}] - \mathbb{E}[B_\B p]\, \mathbb{E}[B_{\B p'}] \right) 
        \le \frac{\Sigma_1 + \Sigma_2}{\mathbb{E}[B]^{2}}
        = o\left(1\right).
    \end{align*}
    This completes the proof of Theorem~\ref{THM:poisson-process-lower-bound-isotropic}. 
\end{proof}

We are now ready to present the proof of Theorem~\ref{THM:poisson-process-lower-bound}. 
We will use the following estimate, which follows from the fact that $\Gamma(x+1) = (1 + o(1)) \sqrt{2\pi x}\left(x/\mathrm{e}\right)^{x}$ as $x\to\infty$ (see e.g.~\cite{Dav59}): 
\begin{align}\label{equ:nu-n-estimate}
    \nu_{n}
    = \frac{\pi^{\frac{n}{2}}}{\Gamma\left(\frac{n}{2} + 1\right)}
    = \frac{\pi^{\frac{n}{2}}}{(1 + o(1)) \sqrt{\pi n}\left(\frac{n}{2\mathrm{e}}\right)^{\frac{n}{2}}}
    = (1 + o(1)) \frac{(2\pi \mathrm{e})^{\frac{n}{2}}}{\sqrt{\pi n}} n^{-\frac{n}{2}}. 
\end{align}

\begin{proof}[Proof of Theorem~\ref{THM:poisson-process-lower-bound}]
    % Let $B$ denote the $n$-dimensional ball of volume $1$ centered at the origin. 
    % It is clear that $B$ is isotropic. 
    By applying an affine transformation to $K$ (and to the lattice defining the torus $\mathbb{T}$), we can assume that $K \subseteq \mathbb{R}^{n}$ is an isotropic body. In particular, $\vol(K) = 1$ by definition. By monotonicity, it suffices to prove Theorem~\ref{THM:poisson-process-lower-bound} for some $\rho\ge n\ln n/2 - (1+o(1)) n \ln\ln n$. Let $\mathbb{T}$ be a packing torus of $K-K$, and let $P$ be a maximal $B_{1/2}^{n}$-packing of $\mathbb{T}$.
    %By~\eqref{equ:nu-n-estimate}, we trivially obtain the upper bound 
    %\begin{align}\label{equ:P-packing-upper-bound}
     %   \ln |P|
    %    \le \ln\left(\frac{\vol(\mathbb{T})}{\vol(B_{1/2}^{n})}\right)
     %   & = \ln\left(\frac{\vol(\mathbb{T})}{(1/2)^{n} \nu_{n}}\right) \notag \\
        % & = (1+o(1)) \ln\left(\frac{8^{n}}{(1/2)^{n} \nu_{n}}\right) \\
        % & = (1+o(1)) \left(\frac{n}{2} \ln n + n \ln\left(\frac{16}{\sqrt{2\me \pi}}\right) \right)
    %    & = \ln\left( (1+o(1)) \frac{2^{n}\cdot \vol(\mathbb{T}) \sqrt{\pi n}}{(2\pi \me)^{\frac{n}{2}}} n^{\frac{n}{2}} \right) \notag \\
     %   & \le \frac{n}{2} \ln n +\ln \left(\vol(\mathbb{T})\right)+ n \ln\left(\frac{2}{\sqrt{2\pi \me}}\right) + \ln\left(2\sqrt{\pi n}\right). 
    %\end{align}
    %
    %On the other hand, it follows from the maximality that $P$ is a unit ball covering of $\mathbb{T}$. 
    
    Since $P$ is maximal, it is a unit ball covering of $\mathbb{T}$.
    Therefore, by~\eqref{equ:nu-n-estimate}, 
    \begin{align}\label{equ:P-packing-lower-bound}
        \ln |P|
        \ge \ln \left( \frac{\vol(\mathbb{T})}{\nu_{n}} \right)
        % = 8^n \frac{\Gamma\left(\frac{n}{2}+1\right)}{\pi^{n/2}}
        % \ge \left(\frac{8}{\sqrt{\pi}}\right)^{n}\left(\frac{n}{2\me}\right)^{\frac{n}{2}}
        % = \left(\frac{8}{\sqrt{2\me \pi}}\right)^{n} n^{\frac{n}{2}},  
        & = \ln \left( (1+o(1))\frac{\vol(\mathbb{T})\cdot \sqrt{\pi n}}{(2\pi \me)^{\frac{n}{2}}} n^{\frac{n}{2}} \right) \notag \\
        & \ge \frac{n}{2} \ln n + \ln \left(\vol(\mathbb{T})\right)+n \ln\left(\frac{1}{\sqrt{2 \pi\me}}\right) + \ln\left(\frac{\sqrt{\pi n}}{2}\right). 
    \end{align}
    %
    %Combining~\eqref{equ:P-packing-upper-bound}, ~\eqref{equ:P-packing-lower-bound} and the fact that $\vol(\mathbb{T})\ge \vol(K-K)\ge 1$, we obtain 
    Let $g_n(t)\coloneqq t-n\ln\ln t$, and note that this function is non-decreasing when $t\ln t\ge n$. In particular, this holds when $t$ is at least the right-hand side of~\eqref{equ:P-packing-lower-bound}. Hence, by~\eqref{equ:P-packing-lower-bound} and $\mathrm{vol}(\mathbb{T})\ge \mathrm{vol}(K-K)\ge 1$, we obtain
    $$
        \ln |P| - n \ln\ln\ln |P|
        \ge \frac{n}{2} \ln n - (1 + o(1)) n \ln\ln n. 
    $$
    Let $\omega(n)$ be a function tending to infinity sufficiently slowly, with $\omega(n) = o(\ln\ln n)$. Let $\mathbf{X}$ be a random set sampled from the Poisson point process on $\mathbb{T}$ with intensity 
    \begin{align*}
        % \ln |P| - n \ln\ln n - \omega(n)
        % = \frac{n}{2} \ln n + n \ln\left(\frac{8}{\sqrt{2\me \pi}}\right) - n \ln\ln n - \omega(n). 
        \rho 
        \coloneqq \ln |P| - n \big(\ln\ln\ln |P| + \omega(n) \big)
        \ge \frac{n}{2}\ln n - (1+o(1)) n \ln\ln n.  
    \end{align*} 
    It follows from Theorem~\ref{THM:poisson-process-lower-bound-isotropic} that 
    \begin{align*}
        \mathbb{P}\left[\,\mathbb{T} \subseteq \mathbf{X} + K\,\right]
        \le \mathbb{P}\left[\,P \subseteq \mathbf{X} + K\,\right]
        = o(1), 
    \end{align*}
    as desired.
    %
    %Rescaling both $K$ and $\mathbb{R}^{n}$ by a factor of $\nu_{n}$, we obtain  Theorem~\ref{THM:poisson-process-lower-bound}. 
\end{proof}

The proof of Theorem~\ref{THM:poisson-process-lower-bound-cube} is similar to that of Theorem~\ref{THM:poisson-process-lower-bound} and will follow from the following counterpart of Theorem~\ref{THM:poisson-process-lower-bound-isotropic}. 
\begin{theorem}\label{THM:hypercube}
    Let $C_{n} = [-1/2,\, 1/2]^{n}$, and let $\mathbb{T}$ be a packing torus of $2 C_n$. 
    Let $P \subseteq \mathbb{T}$ be an $S_n$-packing with $|P|>\vol(\mathbb{T})$, where $S_n$ is the $n$-dimensional cross-polytope
    \begin{align*}
        S_{n}
        \coloneqq \left\{ (x_1, \ldots, x_n) \in \mathbb{R}^{n} \colon |x_1| + \cdots + |x_n| \le 2 \ln n \right\}. 
    \end{align*}
    Suppose that $\mathbf{X} \subseteq \mathbb{T}$ is sampled from the Poisson point process with intensity 
    \begin{align*}
        \rho
        \coloneqq \ln |P| - \ln \left(\vol(\mathbb{T})\right) . 
    \end{align*}
%    where $\omega(n)$ is an arbitrary function tending to infinity as $n\to \infty$, and such that $\rho >0$.  
    Then $\mathbb{P}\left[\,P \subseteq C_{n} + \mathbf{X}\,\right]$ tends to $0$ as $n \to \infty$. 
\end{theorem}

\begin{proof}[Proof of Theorem~\ref{THM:hypercube}]
    Let $\mathbb{T}$, $P$, and $\mathbf{X}$ be as in the theorem. 
    Similar to the proof of Theorem~\ref{THM:poisson-process-lower-bound-isotropic}, for every point $\mathbf{p} \in P$, let $B_{\mathbf{p}}$ be the characteristic function of the event that $\B p \notin \mathbf{X} + C_{n}$. 
    % that is, 
    % \begin{align*}
    %     B_{\mathbf{p}}
    %     = 
    %     \begin{cases}
    %         1, &\quad\text{if}\quad \B p \notin \mathbf{X} + K, \\
    %         0, &\quad\text{if}\quad \B p \in \mathbf{X} + K.
    %     \end{cases}
    % \end{align*} 
    %Similar to~\eqref{equ:E-B}, the sum $B = \sum_{\mathbf{p} \in P} B_{\mathbf{p}}$ satisfies $\mathbb{E}[B] \to \infty$ as $n \to \infty$. 
    Since $\vol(C_n) = 1$, the random variable $|\mathbf{X} \cap (\mathbf{p} - C_n)|$ has the Poisson distribution with mean $\rho \cdot \vol(\mathbf{p} - C_n) = \rho$. 
    Therefore, $\mathbb{E}[B_{\mathbf{p}}] = \me^{-\rho}$.
    Since $\vol(\mathbb{T})\ge \vol(2C_n)=2^n$, it follows that the sum $B \coloneqq \sum_{\mathbf{p} \in P}B_{\mathbf{p}}$ satisfies 
    \begin{align}
        \mathbb{E}[B] 
        = |P|\, \me^{-\rho}
        \ge \vol(\mathbb{T})
        \to \infty
        \quad\text{as}\quad 
        n \to \infty. 
    \end{align}
 
    Let 
    \begin{align*}
        \Sigma_{1}
        & \coloneqq \sum_{\mathbf{p} \in P} \mathbb{E}\left[B_{\mathbf{p}} B_{\mathbf{p}}\right]
        = \sum_{\mathbf{p} \in P} \mathbb{E}\left[B_{\mathbf{p}}\right]
        = \mathbb{E}\left[B\right] \quad\text{and}\quad  \\
        \Sigma_{2}
        & \coloneqq \sum_{\mathbf{p}, \mathbf{p}' \in P,~\mathbf{p} \neq \mathbf{p}'} \left( \mathbb{E}\left[B_{\mathbf{p}} B_{\mathbf{p}'}\right] - \mathbb{E}\left[B_{\mathbf{p}}\right] \mathbb{E}\left[ B_{\mathbf{p}'}\right] \right). 
    \end{align*}
    \begin{claim}\label{CLAIM:Sigma2-cube}
        We have $\Sigma_{2} = o\big( \mathbb{E}\left[B\right]^{2} \big)$. 
    \end{claim}
    \begin{proof}[Proof of Claim~\ref{CLAIM:Sigma2-cube}]
        Fix two distinct points $\mathbf{p}, \mathbf{p}' \in P$. 
        Since $P$ is an $S_{n}$-packing, we have 
        \begin{align*}
            \left(\mathbf{p} + S_{n}\right) \cap \left(\mathbf{p}' + S_{n}\right) 
            = \emptyset,  
        \end{align*}
        which implies that $\mathbf{x} = (x_1, \ldots, x_n) \coloneqq \mathbf{p} - \mathbf{p}'$ does not belong to $ \mathrm{int}(S_{n} - S_{n})$, the interior of $S_{n} - S_{n}$. 
        In particular, $\mathbf{x} \not\in \mathrm{int}(S_n) \subseteq \mathrm{int}(S_n - S_n)$ (since $\mathbf{0} \in S_{n}$), and hence, 
        \begin{align}\label{equ:L1-norm-p-p}
            |x_1| + \cdots + |x_n|
            \ge 2 \ln n.
        \end{align}
        %
%         Similarly to the proof of Theorem~\ref{THM:poisson-process-lower-bound-isotropic}, the volume of $C_{n} \cap (C_{n} + \mathbf{x})$ on $\mathbb{R}^n$ is preserved under the quotient map $\mathbb{R}^n\to \mathbb{T}$. 
Combining it with the AM–GM inequality, we obtain
        \begin{align}\label{equ:cube-intersection-small}
            \vol\left(C_{n} \cap (C_{n} + \mathbf{x})\right)
            = \prod_{i \in [n]} \max\left\{ 1-|x_i|,~0 \right\}
            & \le \max\left\{\left(1 - \frac{|x_1| + \cdots + |x_n|}{n}\right)^{n},0\right\} \notag \\
            & \le  \left(1- \frac{2\ln n}{n}\right)^n 
            \le \frac{1}{n^2}. 
        \end{align}
        For all sufficiently large $n$, we have $B^n_{1/n}\subseteq S_n\cap 2C_n.$ Since $\mathbb{T}$ is a packing torus of $2C_n$, the projection of $B^n_{1/n}$ to $\mathbb{T}$ is injective. Since $P$ is an $S_n$-packing, the projected translates $\mathbf{p}+B^n_{1/n}, \mathbf{p}\in P$, are pairwise disjoint. Therefore,
\begin{align}\label{equ:|P|-upper-bound-cube}
        |P|n^{-n}\nu_n \le \mathrm{vol}(\mathbb{T}).
        \end{align}
        Consequently,
        \begin{align}\label{equ:rho-upper-bound-cube}\rho=\ln|P|-\ln\mathrm{vol}(\mathbb{T})\le n\ln n-\ln\nu_n=O(n\ln n)=o(n^2). 
        \end{align}
        Similarly to~\eqref{equ:exp-Bp-Bp'}, it follows from~\eqref{equ:cube-intersection-small} and~\eqref{equ:rho-upper-bound-cube} that 
        \begin{align*}
            \Sigma_{2}
            & \le \sum_{\mathbf{p} \in P} \sum_{\mathbf{p}' \in P \setminus \{\mathbf{p}\}}  \left( \mathbb{E}\left[B_{\mathbf{p}} B_{\mathbf{p}'}\right] - \mathbb{E}\left[B_{\mathbf{p}}\right] \mathbb{E}\left[ B_{\mathbf{p}'}\right] \right) \\
            & \le |P|^2 \left( \mathrm{exp}\left( -2\rho + \frac{\rho}{n^2} \right) - \mathrm{exp}(-2\rho)\right) 
            = \mathbb{E}\left[B\right]^{2} \left(\mathrm{exp}\left(\frac{\rho}{n^2}\right) - 1\right)
            = o\big( \mathbb{E}\left[B\right]^{2} \big),  
        \end{align*}
        which proves Claim~\ref{CLAIM:Sigma2-cube}. 
    \end{proof}%CLAIM

    It follows from Lemma~\ref{LEMMA:second-moment} that 
    \begin{align*}
        \mathbb{P}[\,B = 0\,]
        \le \frac{\mathrm{Var}[B]}{\mathbb{E}[B]^{2}}
        = \frac{1}{\mathbb{E}[B]^{2}} \sum_{\B p, \B p' \in P} \left( \mathbb{E}[B_\B p B_{\B p'}] - \mathbb{E}[B_\B p]\, \mathbb{E}[B_{\B p'}] \right) 
        \le \frac{\Sigma_1 + \Sigma_2}{\mathbb{E}[B]^{2}}
        = o\left(1\right), 
    \end{align*}
    which proves that $\mathbb{P}\left[\,P \subseteq C_{n} + \mathbf{X}\,\right]$ tends to $0$ as $n \to \infty$. 
    \end{proof}
    
    \begin{proof}[Proof of Theorem~\ref{THM:poisson-process-lower-bound-cube}]
   By monotonicity, it suffices to prove Theorem~\ref{THM:poisson-process-lower-bound-cube} for some $\rho\ge n\ln n - (1+o(1)) n \ln\ln n$. Now suppose that $P$ is a maximal $S_{n}$-packing of $\mathbb{T}$. Since $S_{n}$ is centrally symmetric, it follows that $P$ is a $2S_{n}$-covering of $\mathbb{T}$. 
    Consequently, 
    \begin{align*}
        |P|
        \ge \frac{\vol(\mathbb{T})}{\vol(2 S_{n})}
        = \frac{\vol(\mathbb{T})}{2^n \frac{2^n}{n!} \left(2 \ln n\right)^{n}}
        = (1+o(1))\sqrt{2\pi n} \left(\frac{n}{8 \me \ln n}\right)^n\cdot \vol(\mathbb{T}). 
    \end{align*}
    It follows that 
    \begin{align*}
        \rho 
        = \ln |P| -\ln \left(\vol(\mathbb{T})\right)
        \ge n \ln n - (1+o(1)) n \ln\ln n. 
    \end{align*}
    %
    %Combining it with~\eqref{equ:rho-upper-bound-cube}, we obtain that $\rho = n \ln n - (2+o(1)) n \ln\ln n$. 
    By Theorem~\ref{THM:hypercube},
     \begin{align*}
        \mathbb{P}\left[\,\mathbb{T} \subseteq \mathbf{X} + C_n\,\right]
        \le \mathbb{P}\left[\,P \subseteq \mathbf{X} + C_n\,\right]
        = o(1).
    \end{align*}
    This completes the proof of Theorem~\ref{THM:poisson-process-lower-bound-cube}. 
\end{proof}

%%%%%%%%%%%%%%%%%%%%%%%%%%%%%%%
% \section{Concluding remarks}\label{SEC:Remarks}
%%%%%%%%%%%%%%%%%%%%%%%%%%%%%%%%%%%%%%%%%%%%%%
\bibliographystyle{abbrv}%abbrv
\bibliography{SphereC}
%%%%%%%%%%%%%%%%%%%%%%%%%%%%%%%%%%%%%%%%%%%%%%
\end{document}